\documentclass[11pt,reqno]{amsart}

\newtheorem{proposition}{Proposition}[section]
\newtheorem{lemma}[proposition]{Lemma}
\newtheorem{corollary}[proposition]{Corollary}
\newtheorem{theorem}[proposition]{Theorem}

\theoremstyle{definition}
\newtheorem{definition}[proposition]{Definition}
\newtheorem{example}[proposition]{Example}
\newtheorem{examples}[proposition]{Examples}
\newtheorem{remark}[proposition]{Remark}

\newcommand{\thlabel}[1]{\label{th:#1}}
\newcommand{\thref}[1]{Theorem~\ref{th:#1}}
\newcommand{\selabel}[1]{\label{se:#1}}
\newcommand{\seref}[1]{Section~\ref{se:#1}}
\newcommand{\lelabel}[1]{\label{le:#1}}
\newcommand{\leref}[1]{Lemma~\ref{le:#1}}
\newcommand{\prlabel}[1]{\label{pr:#1}}
\newcommand{\prref}[1]{Proposition~\ref{pr:#1}}
\newcommand{\colabel}[1]{\label{co:#1}}
\newcommand{\coref}[1]{Corollary~\ref{co:#1}}
\newcommand{\relabel}[1]{\label{re:#1}}

\newcommand{\exlabel}[1]{\label{ex:#1}}
\newcommand{\exref}[1]{Example~\ref{ex:#1}}
\newcommand{\delabel}[1]{\label{de:#1}}
\newcommand{\deref}[1]{Definition~\ref{de:#1}}
\newcommand{\eqlabel}[1]{\label{eq:#1}}
\newcommand{\equref}[1]{(\ref{eq:#1})}

\def\ot{\otimes}

\newcommand{\Cc}{\mathcal{C}}

\def\*C{{}^*\hspace*{-1pt}{\Cc}}
\def\text#1{{\rm {\rm #1}}}

\input xy
\xyoption {all} \CompileMatrices

\usepackage{amssymb}
\usepackage{color,amssymb,graphicx,amscd,amsmath}
\usepackage[colorlinks,urlcolor=blue,linkcolor=blue,citecolor=blue]{hyperref}

\begin{document}

\title[Algebraic constructions for Jacobi-Jordan algebras]
{Algebraic constructions for Jacobi-Jordan algebras}

\author{A. L. Agore}
\address{Simion Stoilow Institute of Mathematics of the Romanian Academy, P.O. Box 1-764, 014700 Bucharest, Romania}
\address{Vrije Universiteit Brussel, Pleinlaan 2, B-1050 Brussels, Belgium}
\email{ana.agore@vub.be and ana.agore@gmail.com}

\author{G. Militaru}
\address{Faculty of Mathematics and Computer Science, University of Bucharest, Str.
Academiei 14, RO-010014 Bucharest 1, Romania}
\address{Simion Stoilow Institute of Mathematics of the Romanian Academy, P.O. Box 1-764, 014700 Bucharest, Romania}
\email{gigel.militaru@fmi.unibuc.ro and gigel.militaru@gmail.com}

\thanks{This work was supported by a grant of the Ministry of Research,
Innovation and Digitization, CNCS/CCCDI--UEFISCDI, project number
PN-III-P4-ID-PCE-2020-0458, within PNCDI III. The first author is a fellow of FWO Vlaanderen.}


\begin{abstract}
For a given Jacobi-Jordan algebra $A$ and a vector space $V$ over a
field $k$, a non-abelian cohomological type object ${\mathcal
H}^{2}_{A} \, (V, \, A)$ is constructed: it classifies all
Jacobi-Jordan algebras containing $A$ as a subalgebra of codimension
equal to ${\rm dim}_k (V)$. Any such algebra is isomorphic to a so-called
\emph{unified product} $A \, \natural \, V$. Furthermore, we introduce the bicrossed
(semi-direct, crossed, or skew crossed) product $A \bowtie V$
associated to two Jacobi-Jordan algebras as a special
case of the unified product. Several examples and applications
are provided: the Galois group of the extension $A \subseteq A
\bowtie V$ is described as a subgroup of the semidirect product of
groups ${\rm GL}_k (V) \rtimes {\rm Hom}_k (V, \, A)$ and an Artin
type theorem for Jacobi-Jordan algebra is proven. 
\end{abstract}

\maketitle

\section*{Introduction}
In the jungle of non-associative algebras, \emph{Jacobi-Jordan algebras} (JJ algebras for short)
are rather special objects. A JJ algebra is a vector space $A$
together with a bilinear map $[- , \,  -] : A \times A \to A$ such
that $[a, \, b] = [b, \, a]$ and $[a, \, [b, \, c] ] + [b, \, [c,
\, a] ] + [c, \, [a, \, b] ]  = 0$, for all $a$, $b$, $c\in A$.
According to \cite{zu2}, where a detailed history of JJ algebras is given and
several conjectures are proposed, this family
of algebras was first defined in \cite{Zh} and since then they have been studied independently in various papers \cite{BB0, BB, BF, Camacho,
GK, Ok, wa, Buse, Zh2} under different names such as \emph{Lie-Jordan
algebras}, \emph{Jordan algebras of nil index $3$}, \emph{pathological
algebras}, \emph{mock-Lie algebras} or \emph{Jacobi-Jordan algebras}. Throughout, we shall adopt the name
JJ algebras. Although at first sight JJ algebras are very close to Lie algebras having only the
skew-symmetry condition replaced by the symmetry condition, we will see that in fact this class of algebras exhibits rather different properties. Indeed, for instance two classical theorems in Lie algebra theory, namely Ado's theorem and the
Poincar\'{e}--Birkhoff--Witt theorem, fail for JJ algebras
\cite{zu2}. However, JJ algebras have an
interesting and rich structure theory which deserves to be developed further. This is the staring point of this paper which is organised as follows. The first section fixes notations and conventions used
throughout and recalls some basic concepts in the context of
JJ algebras. \seref{unifiedprod} is devoted
to the study of the \emph{extending structures problem}
(ES-problem), introduced in \cite{am-2011} for arbitrary
categories. In the context of  JJ algebras it comes down to the following
question:

\textit{Let $A$ be a JJ algebra and $E$ a vector space
containing $A$ as a subspace. Describe and classify up to an
isomorphism that stabilizes $A$ (i.e. acts as the identity on $A$)
the set of all JJ algebra structures that can be defined on
$E$ such that $A$ becomes a subalgebra of $E$.}

If we fix $V$ a complement of $A$ in the vector space $E$ then the
ES-problem asks for the description and classification of all JJ algebras containing and stabilizing $A$ as a
subalgebra of codimension equal to ${\rm dim} (V)$. Following the
strategy we previously developed in \cite{am-2014, am-2019} the approach we will use for studying the ES-problem
is the following: we start by constructing in \thref{1} the \emph{unified product} $A \, \natural \, V$
associated to a JJ algebra $A$ and a vector space $V$
connected through two \emph{actions} and a \emph{cocycle}. Next we show in \thref{classif} that a JJ algebra structure $(E, \,
[-, \, -]_E )$ on $E$ contains $A$ as a subalgebra if and only if
there exists an isomorphism of JJ algebras $(E, \, [-, \,
-]_E ) \cong A \, \natural \, V$. Finally, the theoretical answer
to the ES-problem is given in \thref{main1}: a \emph{non-abelian
cohomological type object} ${\mathcal H}^{2}_{A} \, (V, \, A )$ is
explicitly constructed; it parameterizes and classifies all
JJ algebras containing and stabilizing $A$ as a subalgebra of
codimension equal to ${\rm dim} (V)$. The unified product is a
general construction containing as special cases the bicrossed
product, semi-direct product, crossed product or skew crossed
product associated to JJ algebras. \seref{cazurispeciale} describes in detail all these special cases, highlighting the role and scope of the subsequent problem associated to each such product. For instance, in \deref{mpmL} we introduce matched pairs of JJ algebras and the corresponding bicrossed product: these are the JJ counterparts of similar
constructions performed for Lie algebras \cite[Theorem 4.1]{majid}.
\coref{bicromlver} proves that the bicrossed product of two
JJ algebras is the object responsible for the
\emph{factorization problem} and is the JJ algebra version of
\cite[Theorem 3.9]{LW}. If $A \bowtie V$ is the bicrossed product
associated to a matched pair $(A, V, \triangleleft, \,
\triangleright)$ of JJ algebras, then the Galois group of
the extension $A \subseteq A \bowtie V$ is explicitly computed
in \coref{grupulgal} as a subgroup of the semidirect product of
groups ${\rm GL}_k (V) \rtimes {\rm Hom}_k (V, \, A)$. The crossed product of two JJ algebras is also a special case of the
unified product: it was introduced and studied in \cite{am-2015}
related to Hilbert's extension problem. Here we highlight a
new application of crossed products as the main characters in our strategy for classifying finite
dimensional supersolvable JJ algebras (\prref{crossdim1}).
Skew crossed products are also introduced and used in
\thref{recsiGal} in order to prove an Artin type theorem for
JJ algebras which gives the reconstruction of a JJ
algebra $A$, on which the finite group $G$ acts, from its
subalgebra of invariants. Computing the classifying object
${\mathcal H}^{2}_{A} \, (V, \, A)$ constructed in \thref{main1},
for a given JJ algebra $A$ and a vector space $V$ is, in general, a very difficult problem. As the starting point in achieving this goal we develop in \seref{exemple} a general strategy for explicitly computing ${\mathcal H}^{2}_{A} \, (k, \, A)$ (\thref{clasdim1}). 

\section{Preliminaries}\selabel{prel}
All vector spaces, linear or bilinear maps are over an arbitrary
field $k$. A bilinear map $f : W\times W \to V$ is called
symmetric if $f (x, \, y) = f (y, \, x)$, for all $x$, $y \in W$.
For a vector space $V$ its dual is denoted by $V^* := {\rm Hom}_k
(V, \, k)$ and ${\rm GL}_k (V) := {\rm Aut}_k (V)$ is the group of
all linear automorphisms of $V$.

A \emph{JJ algebra} is a vector space $A$ together with a
bilinear map $[- , \,  -] : A \times A \to A$, called
multiplication, such that for any $a$, $b$, $c\in A$:
\begin{equation}\eqlabel{jjdef}
[a, \, b]  = [b, \, a], \qquad [a, \, [b, \, c] ] + [b, \, [c, \,
a] ] + [c, \, [a, \, b] ]  = 0
\end{equation}
that is, $[- , \, -]$ is commutative/symmetric and satisfies the
Jacobi identity. Throughout, when describing the multiplication of a certain JJ algebra we will only write down the non-zero products.
The concepts of morphism of JJ algebras,
subalgebra, ideal, derivation etc. are defined in the obvious
way. We denote by ${\rm Aut}_{\rm JJ}(A)$ the automorphism group of the JJ algebra $A$. $\sum_{(c)}$ stands for the circular sum: for
example, $\sum_{(c)} \, [a, \, [b, \, c] ] = [a, \, [b, \, c] ] +
[b, \, [c, \, a] ] + [c, \, [a, \, b] ]$. A JJ algebra $A$
is called abelian if it has trivial multiplication, i.e.
$[a, \,  b] = 0$ for all $a$, $b\in A$. Over a field of
characteristic $\neq 2, 3$, any JJ algebra is a Jordan
algebra \cite{BF} and Jordan algebras of nilpotent index $3$ are
JJ algebras \cite{zu2}. If $A$ is a JJ algebra and $B$ a
commutative associative algebra, then the tensor product $A \ot B$ can be endowed with a JJ algebra structure whose multiplication is given by the formula:
$$[a \ot b, \, a' \ot
b'] : = [a, \, a'] \ot bb'$$
for all $a, a'\in A$ and $b$, $b'\in
B$. Following \cite{zu}, we call this object a \emph{current JJ
algebra}. Several examples of JJ algebras are given in
\cite{am-2015, BB, BF, zu2}. An \emph{antiderivation} of a
JJ algebra $A$ is a linear map $D: A \to A$ such that for
any $a$, $b\in A$:
$$
D ([a, \, b]) = - [D(a), \, b] - [a, \, D(b)]
$$
Unfortunately, unlike Lie algebras, the space of all derivations (resp. antiderivations) of a JJ algebra does not carry a canonical JJ algebra structure.

A \emph{left JJ $A$-module} \cite[Definition 1.4]{am-2015}
is a vector space $V$ equipped with a bilinear map $\triangleright
: A \times V \to V$, called action, such that for any $a$, $b\in
A$ and $x\in V$:
\begin{eqnarray}
[a, \, b] \triangleright x = - a \triangleright (b \triangleright
x) - b \triangleright (a \triangleright x) \eqlabel{JJmodleft}
\end{eqnarray}
We denote by ${}_A{\mathcal M}$ the category of all (left)
JJ $A$-modules with action preserving linear maps as
morphisms. A \emph{right JJ $A$-module} is a vector space
$V$ equipped with a bilinear map $\triangleleft : V \times A \to
V$ such that for any $a$, $b\in A$ and $x\in V$:
\begin{eqnarray}
x \triangleleft [a, \, b]  = - (x \triangleleft a) \triangleleft b
- (x \triangleleft b) \triangleleft a  \eqlabel{JJmodrigt}
\end{eqnarray}
Since $A$ is in particular a commutative algebra there exists an isomorphism of categories
${\mathcal M}_A \cong {}_A{\mathcal M}$. Nevertheless, we shall use
both categories throught the paper. We can easily see that $A$ and
the linear dual $A^* := {\rm Hom}_k (A, k)$ are left JJ
$A$-modules via the canonical actions:
\begin{equation}\eqlabel{canact}
a \triangleright x := [a , \,  x], \qquad (a \triangleright a^*)
(x) := a^* ([a , \,  x])
\end{equation}
for all $a$, $x\in A$ and $a^* \in A^*$.

\section{Extending structures problem}\selabel{unifiedprod}

This section deals with the extending structures problem for
JJ algebras. First we introduce the following:

\begin{definition} \delabel{echivextedn}
Let $A$ be a JJ algebra and $E$ a vector space containing $A$ as a subspace. Two JJ algebra structures on $E$, $\{-,
\, -\}$ and $\{-, \, -\}'$, both containing $A$ as a subalgebra, are called \emph{equivalent}, and we denote this by
$(E, \{-, \, -\}) \equiv (E, \{-, \, -\}')$, if there exists a
JJ algebra isomorphism $\varphi: (E, \{-, \, -\}) \to (E,
\{-, \, -\}')$ which stabilizes $A$, that is $\varphi (a) = a$,
for all $a\in A$. We denote by ${\rm Extd} \, (E, A)$ the set of all equivalence
classes on the set of all JJ algebras structures on $E$
containing $A$ as a subalgebra via the equivalence relation
$\equiv$.
\end{definition}

${\rm Extd} \, (E, A)$ as defined above is the
classifying object of the extending structures problem. We shall
prove that ${\rm Extd} \, (E, A)$ is parameterized by a
\emph{cohomological type object,} denoted by ${\mathcal
H}^{2}_{A} \, (V, \, A)$, which will be explicitly constructed in
this section, where $V$ is a complement of $A$ in $E$, that is $E
= A + V$ and $A \cap V = 0$.

\begin{definition}\delabel{exdatum}
Let $A$ be a JJ algebra and $V$ a vector space. An
\textit{extending datum of $A$ through $V$} is a system $\Omega(A,
\,  V) = \bigl(\triangleleft, \, \triangleright, \, f, \{-, \, -\}
\bigl)$ consisting of four bilinear maps
$$
\triangleleft : V \times A \to V, \quad \triangleright : V \times
A \to A, \quad f: V\times V \to A, \quad \{-, \, -\} : V\times V
\to V.
$$
Let $\Omega(A, \, V) = \bigl(\triangleleft, \, \triangleright, \,
f, \{-, \, -\} \bigl)$ be an extending datum. We denote by $ A \,
\,\natural \,_{\Omega(A, \, V)} V = A \, \,\natural \, V$ the
vector space $A \, \times V$ together with the bilinear map $[ -, \, -] :
(A \times V) \times (A \times V) \to A \times V$ defined by:
\begin{equation}\eqlabel{brackunif}
[(a, x), \, (b, y)] := \bigl( [a, \, b] + x \triangleright b +
y\triangleright a + f(x, y), \,\, \{x, \, y \} + x\triangleleft b
+ y\triangleleft a \bigl)
\end{equation}
for all $a$, $b \in A$ and $x$, $y \in V$. The object $A
\,\natural \, V$ is called the \textit{unified product} of $A$ and
$V$ if it is a JJ algebra with the multiplication given by
\equref{brackunif}. In this case the extending datum $\Omega(A, \,
V) = \bigl(\triangleleft, \, \triangleright, \, f, \{-, \, -\}
\bigl)$ is called a \textit{JJ extending structure} of $A$
through $V$. The maps $\triangleleft$ and $\triangleright$ are
called the \textit{actions} of $\Omega(A, \, V)$ and $f$ is called
the \textit{cocycle} of $\Omega(A, \, V)$.
\end{definition}

Let $\Omega(A, \, V)$ be an extending datum of $A$ through $V$.
Then, the following relations, very useful in computations, hold
in $A \,\natural \, V$:
\begin{eqnarray}
[(a, 0), \, (b, y)] &=& \bigl([a, \, b] + y \triangleright
a, \,\,  y \triangleleft a \bigl) \eqlabel{001}\\
\left[(0, x), \, (b, y)\right] &=& \bigl( x \triangleright b +
f(x, y), \,\, x \triangleleft b + \{x, \, y\} \bigl) \eqlabel{002}
\end{eqnarray}
for all $a$, $b \in A$ and $x$, $y \in V$.

\begin{theorem}\thlabel{1}
Let $\Omega(A, \,  V) = \bigl(\triangleleft, \, \triangleright, \,
f, \{-, \, -\} \bigl)$ be an extending datum of a JJ algebra
$A$ through a vector space $V$. The following statements are
equivalent:

$(1)$ $A \,\natural \, V$ is a unified product;

$(2)$ The following compatibilities hold for any $a$, $b \in A$,
$x$, $y$, $z \in V$:
\begin{enumerate}
\item[(E1)] $f: V\times V \to A$ and $\{-, \, -\} : V\times V
\to V$ are symmetric maps; \\
\item[(E2)] $(V, \, \triangleleft)$ is a right JJ $A$-module;\\
\item[(E3)] $x \triangleright [a, \, b] = - [x \triangleright a,
\, b] - [a, \, x \triangleright b] - (x \triangleleft a)
\triangleright b - (x \triangleleft b) \triangleright a$;\\
\item[(E4)] $\{x,\, y \} \triangleleft a = - \{x,\, y
\triangleleft a\} - \{x \triangleleft a, \, y \} - x \triangleleft
(y \triangleright a)
- y \triangleleft (x \triangleright a)$;\\
\item[(E5)] $\{x,\, y \} \triangleright a =  - x \triangleright (y
\triangleright a) - y \triangleright (x \triangleright a) - [a, \,
f(x, \, y)] - f(x, \, y \triangleleft a) - f(x \triangleleft a, \, y)$;\\
\item[(E6)] $ \sum_{(c)} \, f\bigl(x, \{y,\, z \}\bigl) +
\sum_{(c)} \, x \triangleright f(y, z) = 0$;\\
\item[(E7)] $ \sum_{(c)} \, \{x, \, \{y, \,z\}\} + \sum_{(c)} \, x
\triangleleft f(y, z) = 0$.
\end{enumerate}
\end{theorem}

\begin{proof}
The proof is based on a rather long and laborious but straightforward computation. We restrict to indicating only the main steps of the proof. First, we can easily prove that the multiplication
defined by \equref{brackunif} is commutative if and only if $f:
V\times V \to A$ and $\{-, \, -\} : V\times V \to V$ are both
symmetric maps, i.e. (E1) holds. From now on we will assume that
(E1) holds. Thus $A \,\natural \, V$ is a JJ algebra if and
only if Jacobi's identity holds, i.e.:
\begin{equation}\eqlabel{005}
\sum_{(c)} \, \bigl[(a, x), \, [(b, y), \, (c, z)]\bigl] = 0
\end{equation}
for all $a$, $b$, $c \in A$ and $x$, $y$, $z \in V$. Since in $A
\,\natural \, V$ we have $(a, x) = (a, 0) + (0, x)$ it follows
that \equref{005} holds if and only if it holds for all generators
of $A \,\natural \, V$, i.e. for the set $\{(a, \, 0) ~|~ a \in A\}
\cup \{(0, \, x) ~|~ x \in V\}$. Since \equref{005} is invariant
under circular permutations we are left with only three cases to
study. First, using \equref{001} we can easily notice that
\equref{005} holds for the triple $(a, 0)$, $(b, 0)$, $(c, 0)$.
Next, we can prove that \equref{005} holds for $(a, 0)$, $(b, 0)$,
$(0, x)$ if and only if (E2) and (E3) hold. Secondly, we can prove
that \equref{005} holds for $(a, 0)$, $(0, x)$, $(0, y)$ if and
only if (E4) and (E5) hold. Finally, \equref{005} holds for $(0,
x)$, $(0, y)$, $(0, z)$ if and only if (E6) and (E7) hold and the
proof is finished.
\end{proof}

For a given JJ algebra $A$ a vector space $V$, we denote by
${\mathcal J}{\mathcal J} (A, \, V)$ the set of all JJ
extending structures of $A$ through $V$, i.e. all systems
$\Omega(A, \, V) = \bigl(\triangleleft, \, \triangleright, \, f,
\{-, \, -\} \bigl)$ satisfying the compatibility conditions
(E1)-(E7) of \thref{1}. Several examples of extending structures
will be given in \seref{cazurispeciale} and \seref{exemple}. Notice that the set ${\mathcal J}{\mathcal J} (A, \, V)$ is
nonempty: it contains the extending structure $\Omega (A, \, V) =
\bigl(\triangleleft, \, \triangleright, \, f, \{-, \, -\} \bigl)$
for which all bilinear maps are trivial. In this case the
associated unified product $A \,\natural \, V = A \times V$, the
direct product between $A$ and the abelian JJ algebra $V$.

Let $\Omega(A, \, V) = \bigl(\triangleleft, \, \triangleright, \,
f, \{-, \, -\} \bigl) \, \in {\mathcal J} {\mathcal J} (A, \, V)$
be a JJ extending structure and $A \,\natural \, V$ the
associated unified product. Then the canonical inclusion
$$
i_{A}: A \to A \,\natural \, V, \qquad i_{A}(a) = (a, \, 0)
$$
is an injective JJ algebra map. Therefore, we can see $A$ as
a JJ subalgebra of $A \,\natural \, V$ through the
identification $A \cong i_{A}(A) \cong A \times \{0\}$.
Conversely, we will prove that any JJ algebra structure on a
vector space $E$ containing $A$ as a JJ subalgebra is
isomorphic to a unified product.

\begin{theorem}\thlabel{classif}
Let $A$ be a JJ algebra, $E$ a vector space containing $A$
as a subspace and $[-, \,-]$ a JJ algebra structure on $E$
such that $A$ is a JJ subalgebra in $(E, [-, \,-])$. Then
there exists a JJ extending structure $\Omega(A, \, V) =
\bigl(\triangleleft, \, \triangleright, \, f, \{-, \, -\} \bigl)$
of $A$ through a subspace $V$ of $E$ and an isomorphism of JJ
algebras $(E, [-, \,-]) \cong A \,\natural \, V$ that stabilizes
$A$.
\end{theorem}

\begin{proof} Since we work over a field $k$, there exists a linear map $p: E \to
A$ such that $p(a) = a$, for all $a \in A$. Then $V :=
\rm{ker}(p)$ is a subspace of $E$ and a complement of $A$ in $E$.
We can now define the extending datum of $A$ through $V$ as follows:
\begin{eqnarray*}
\triangleright = \triangleright_p : V \times A \to
A, \qquad x \triangleright a &:=& p \bigl([x, \, a]\bigl)\\
\triangleleft = \triangleleft_p: V \times A \to V,
\qquad x \triangleleft a &:=& [x, \, a] - p \bigl([x, \, a]\bigl)\\
f = f_p: V \times V \to A, \qquad f(x, y) &:=&
p \bigl([x, \, y]\bigl)\\
\{\, , \, \} = \{\, , \, \}_p: V \times V \to V, \qquad \{x, y\}
&:=& [x, \, y] - p \bigl([x, \, y]\bigl)
\end{eqnarray*}
for any $a \in A$ and $x$, $y\in V$. First of all, it is straightforward to see that
the above maps are well defined bilinear maps: $x
\triangleleft a \in V$ and $\{x, \, y \} \in V$, for all $x$, $y
\in V$ and $a \in A$. We will show that $\Omega(A, \, V) =
\bigl(\triangleleft, \, \triangleright, \, f, \{-, \, -\} \bigl)$
is a JJ extending structure of $A$ through $V$ and $ \varphi:
A \, \natural \, V \to E$, $\varphi(a, x) := a + x$ is an
isomorphism of JJ algebras that stabilizes $A$. The strategy we use, relaying on \thref{1}, is the following: $\varphi: A \times V \to E$,
$\varphi(a, \, x) := a + x$ is a linear isomorphism between the
JJ algebra $E$ and the direct product of vector spaces $A
\times V$ with the inverse given by $\varphi^{-1}(y) :=
\bigl(p(y), \, y - p(y)\bigl)$, for all $y \in E$. Thus, there
exists a unique JJ algebra structure on $A \times V$ such
that $\varphi$ is an isomorphism of JJ algebras and this
unique multiplication on $A \times V$ is given for any $a$, $b \in
A$ and $x$, $y\in V$ by:
$$
[(a, x), \, (b, y)] := \varphi^{-1} \bigl([\varphi(a, x), \,
\varphi(b, y)]\bigl)
$$
We are now left to prove that the above
multiplication coincides with the one associated to the system
$\bigl(\triangleleft_p, \, \triangleright_p, \, f_p, \{-, \, -\}_p
\bigl)$ as defined by
\equref{brackunif} . Indeed, for any $a$, $b \in A$ and $x$, $y\in V$ we have:
\begin{eqnarray*}
[(a, x), \, (b, y)] &=& \varphi^{-1} \bigl([\varphi(a, x), \,
\varphi(b, y)]\bigl)
= \varphi^{-1} \bigl([a, \, b] + [a, \, y] + [x, \, b] + [x, \, y]\bigl)\\
&=& \bigl(p([a, \, b]), [a, \, b] - p([a, \, b])\bigl) +
\bigl(p([a, \, y]), [a, \, y] - p([a, \, y])\bigl)\\
&& + \bigl(p([x, \, b]), [x, \, b] - p([x, \, b])\bigl) +
\bigl(p([x, \, y]), [x, \, y] - p([x, \, y])\bigl)\\
&=& \Bigl(p([a, \, b]) + p([a, \, y]) + p([x, \, b]) +
p([x, \, y]), \ [a, \, b] + [a, \, y]\\
&&+ [x, \, b] + [x, \, y] - p([a, \, b]) - p([a, \, y])
- p([x, \, b]) - p([x, \, y])\Bigl)\\
&=& \Bigl([a, \, b] + p([a, \, y]) + p([x, \, b]) +
p([x, \, y]), \,\, [a, \, y] - p([a, \, y]) + \\
&& + [x, \, b] - p([x, \, b])  + [x, \, y] - p([x, \, y])\Bigl)\\
&=& \Bigl([a, \, b] + y \triangleright a + x \triangleright b +
f(x, \, y), \, \{x, \, y\} + x \triangleleft b + y \triangleleft
a\Bigl)
\end{eqnarray*}
as desired. Note that the commutativity of $[-, \, -]$ was intensively used
in the above computations. Moreover, the following diagram
\begin{eqnarray*}
\xymatrix {& A \ar[r]^{i_{A}} \ar[d]_{Id} &
{A \,\natural \, V} \ar[d]^{\varphi} \\
& A \ar[r]^{i} & {E} }
\end{eqnarray*}
is obviously commutative which shows that $\varphi$ stabilizes $A$ and this finishes the
proof.
\end{proof}

Using \thref{classif}, the classification of all JJ algebra
structures on $E$ that contain $A$ as a subalgebra, reduces to
the classification of all unified products $A \,\natural \, V$,
associated to all JJ extending structures $\Omega(A, \, V) =
\bigl(\triangleleft, \, \triangleright, \, f, \{-, \, -\} \bigl)$,
for a given complement $V$ of $A$ in $E$. In order to construct a
cohomological type object ${\mathcal H}^{2}_{A} \, (V, \, A)$
which will parameterize the classifying sets ${\rm Extd} \, (E,
A)$ defined in \deref{echivextedn}, we introduce the following:

\begin{lemma} \lelabel{morfismuni}
Let $\Omega(A, \, V) = \bigl(\triangleleft, \, \triangleright, \,
f, \{-, \, -\} \bigl)$ and $\Omega'(A, \, V) = \bigl(\triangleleft
', \, \triangleright ', \, f', \{-, \, -\}' \bigl)$ be two
JJ algebra extending structures of $A$ through $V$ and $ A
\,\natural \, V$, respectively $ A \,\natural \, ' V$, the associated unified
products. Then there exists a bijection between the set of all
morphisms of JJ algebras $\psi: A \,\natural \, V \to A
\,\natural \, ' V$ which stabilize $A$ and the set of pairs $(r,
v)$, where $r: V \to A$, $v: V \to V$ are two linear maps
satisfying the following compatibility conditions for any $a \in
A$, $x$, $y \in V$:
\begin{enumerate}
\item[(M1)] $v(x) \triangleleft ' a = v(x \triangleleft a)$, i.e.
$v$ is a morphism of right JJ $A$-modules; \item[(M2)] $
v(x) \triangleright ' a = r(x \triangleleft a) + x \triangleright
a - [a, \, r(x)]$; \item[(M3)] $v(\{x, y\}) = \{v(x), v(y)\}' +
v(x) \triangleleft ' r(y) + v(y) \triangleleft ' r(x)$;
\item[(M4)] $r(\{x, y\}) = [r(x), \, r(y)] + v(x) \triangleright '
r(y) + v(y) \triangleright ' r(x) + f' \bigl(v(x), v(y)\bigl) -
f(x, y)$
\end{enumerate}
Under the above bijection the morphism of JJ algebras $\psi
= \psi_{(r, v)}: A \,\natural \, V \to A \,\natural \, ' V$
corresponding to $(r, v)$ is given for any $a \in A$ and $x \in V$
by:
$$
\psi(a, \, x) = (a + r(x), \, v(x))
$$
Moreover, $\psi = \psi_{(r, v)}$ is an isomorphism if and only if
$v: V \to V$ is bijective.
\end{lemma}

\begin{proof}
A linear map $\psi: A \,\natural \, V \to A \,\natural \, ' V$
which makes the following diagram commutative:
$$
\xymatrix {& {A} \ar[r]^{i_{A}} \ar[d]_{Id_{A}}
& {A \,\natural \, V}\ar[d]^{\psi}\\
& {A} \ar[r]^{i_{A}} & {A \,\natural \, ' V}}
$$
is uniquely determined by two linear maps $r: V \to A$, $v: V \to
V$ such that $\psi(a, x) = (a + r(x), v(x))$, for all $a \in A$,
and $x \in V$. Indeed, if we denote $\psi(0, x) = (r(x), v(x)) \in
A \times V$ for all $x \in V$, we obtain:
\begin{eqnarray*}
\psi(a, \, x) &=& \psi \bigl((a, 0) + \psi(0, x)\bigl) = \psi(a, 0) + \psi(0, x)\\
&=&(a, 0) + \bigl(r(x), v(x)\bigl) = \bigl(a + r(x), v(x) \bigl)
\end{eqnarray*}
Let $\psi = \psi_{(r, v)}$ be such a linear map, i.e. $\psi(a, x)
= (a + r(x), v(x))$, for some linear maps $r: V \to A$, $v: V \to
V$. We will prove that $\psi$ is a morphism of JJ algebras
if and only if the compatibility conditions (M1)-(M4) hold. To this end, it is
enough to prove that the compatibility
\begin{equation}\eqlabel{Liemap}
\psi \bigl([(a, x), \, (b, y)] \bigl) = [\psi(a, x), \, \psi(b,
y)]
\end{equation}
holds on all generators of $A \,\natural \, V$. Again, we skip the detailed computations and indicate only the key steps of the process. First, it is easy to see that
\equref{Liemap} holds for the pair $(a, 0)$, $(b, 0)$, for all
$a$, $b \in A$. Secondly, we can prove that \equref{Liemap} holds
for the pair $(a, 0)$, $(0, x)$ if and only if (M1) and (M2) hold.
Finally, \equref{Liemap} holds for the pair $(0,
x)$, $(0, y)$ if and only if (M3) and (M4) hold. The last
statement follows immediately by noticing that if $v: V \to
V$ is bijective, then $\psi_{(r, v)}$ is an isomorphism of
JJ algebras with the inverse given for any $b \in A$ and $y
\in V$ by:
$$
\psi_{(r, v)}^{-1}(b, \, y) = \bigl(b - r(v^{-1}(y)), \,
v^{-1}(y)\bigl)
$$
The proof is now finished.
\end{proof}

For classification purposes we introduce the following:

\begin{definition}\delabel{echiaa}
Let $A$ be a JJ algebra and $V$ a vector space. Two
JJ algebra extending structures of $A$ by $V$, $\Omega(A, \,
V) = \bigl(\triangleleft, \, \triangleright, \, f, \{-, \, -\}
\bigl)$ and $\Omega'(A, \, V) = \bigl(\triangleleft ', \,
\triangleright ', \, f', \{-, \, -\}' \bigl)$ are called
\emph{equivalent}, and we denote this by $\Omega(A, \, V) \equiv
\Omega'(A, \, V)$, if there exists a pair of linear maps $(r, v)$,
where $r: V \to A$ and $v \in {\rm Aut}_{k}(V)$ such that
$\bigl(\triangleleft ', \, \triangleright ', \, f', \{-, \, -\}'
\bigl)$ is defined via $\bigl(\triangleleft, \,
\triangleright, \, f, \{-, \, -\} \bigl)$ using $(r, v)$ as follows:
\begin{eqnarray*}
x \triangleleft ' a &=& v \bigl(v^{-1}(x) \triangleleft a\bigl)\\
x \triangleright ' a &=& r \bigl(v^{-1}(x) \triangleleft a\bigl) +
v^{-1}(x) \triangleright a - [a, \, r\bigl(v^{-1}(x)\bigl)]\\
f'(x, y) &=& f \bigl(v^{-1}(x), v^{-1}(y)\bigl) + r \bigl(\{v^{-1}(x),
v^{-1}(y)\}\bigl) + [r\bigl(v^{-1}(x)), \, r\bigl(v^{-1}(y)\bigl)]\\
&& - \, r\bigl(v^{-1}(x) \triangleleft r
\bigl(v^{-1}(y)\bigl)\bigl) - v^{-1}(x) \triangleright r
\bigl(v^{-1}(y)\bigl)
-  r\bigl(v^{-1}(y) \triangleleft r \bigl(v^{-1}(x)\bigl)\bigl)\\
&&  - \, v^{-1}(y) \triangleright r \bigl(v^{-1}(x)\bigl) \\
\{x, y\}' &=& v \bigl(\{v^{-1}(x), v^{-1}(y)\}\bigl) - \, v
\bigl(v^{-1}(x) \triangleleft r \bigl(v^{-1}(y)\bigl)\bigl)
 - \, v \bigl(v^{-1}(y) \triangleleft r \bigl(v^{-1}(x)\bigl)\bigl)
\end{eqnarray*}
for all $a \in A$, $x$, $y \in V$.
\end{definition}

We summarize the results of this section in the following result which provides the answer to the extending
structures problem for JJ algebras:

\begin{theorem}\thlabel{main1}
Let $A$ be a JJ algebra, $E$ a vector space that contains
$A$ as a subspace and $V$ a complement of $A$ in $E$. Then:

$(1)$ $\equiv$ is an equivalence relation on the set $ {\mathcal
J} {\mathcal J} (A, \,  V)$ of all JJ extending structures
of $A$ through $V$. We denote by ${\mathcal H}^{2}_{A} \, (V, \, A
) := {\mathcal J} {\mathcal J} (A, \, V)/ \equiv $, the quotient
set.

$(2)$ The map
$$
{\mathcal H}^{2}_{A} \, (V, \, A) \to {\rm Extd} \, (E, \, A),
\qquad \overline{(\triangleleft, \triangleright, f, \{-, \, -\})}
\rightarrow \bigl(\mathfrak{g} \,\natural \, V, \, [ -  , \, - ]
\bigl)
$$
is bijective, where $\overline{(\triangleleft, \triangleright, f,
\{-, \, -\})}$ is the equivalence class of $(\triangleleft,
\triangleright, f, \{-, \, -\})$ via $\equiv$.
\end{theorem}

\begin{proof} The proof follows from \thref{1},
\thref{classif} and \leref{morfismuni} once we observe that
$\Omega(A, \,  V) \equiv \Omega'(A, \, V)$ in the sense of
\deref{echiaa} if and only if there exists an isomorphism of
JJ algebras $\psi: A \,\natural \, V \to A \,\natural \, '
V$ which stabilizes $A$. Therefore, $\equiv$ is an equivalence
relation on the set ${\mathcal J} {\mathcal J} (A, \, V)$ of all
JJ algebra extending structures $\Omega(A, \, V)$ and the
conclusion follows from \thref{classif} and \leref{morfismuni}.
\end{proof}

\section{Special cases of unified products. Applications}\selabel{cazurispeciale}
In this section we consider the most important special cases of unified products of JJ algebras, namely bicrossed/semidirect/crossed/skew
crossed products, and we will provide applications for each of these
products. We consider the following convention: if one of the maps
$\triangleleft$, $\triangleright$, $f$ or $\{-, \, -\}$ of an
extending datum $\Omega(A, \, V) = \bigl(\triangleleft, \,
\triangleright, \, f, \{-, \, -\} \bigl)$ is trivial then we will
omit it from the quadruple $\bigl(\triangleleft, \,
\triangleright, \, f, \{-, \, -\} \bigl)$.

\subsection*{Matched pairs and bicrossed products}
Let $\Omega(A, \, V)=\bigl(\triangleleft, \triangleright, f, \{-,
\, -\}\bigl)$ be an extending datum of the JJ algebra $A$
through a vector space $V$ such that $f$ is the trivial map, i.e.
$f(x, y) = 0$ for all $x$, $y\in V$. Then, using \thref{1} we
obtain that $\Omega(A, \, V) = \bigl(\triangleleft, \,
\triangleright, \, \{-, \, -\} \bigl)$ is a JJ extending
structure of $A$ through $V$ if and only if $(V, \, \{-, -\})$ is
a JJ algebra and the following compatibilities hold for all
$a$, $b \in A$, $x$, $y \in V$:
\begin{enumerate}
\item[(1)] $(V, \, \triangleleft)$ is a right JJ $A$-module,
i.e. $x \triangleleft [a, \, b]  = - (x \triangleleft
a) \triangleleft b - (x \triangleleft b) \triangleleft a$ ;\\
\item[(2)] $(A, \, \triangleright)$ is a left JJ $V$-module,
i.e. $\{x,\, y \} \triangleright a =  - x \triangleright (y
\triangleright a) - y \triangleright (x
\triangleright a)$;\\
\item[(3)] $x \triangleright [a, \, b] = - [x \triangleright a, \,
b] - [a, \, x \triangleright b] - (x \triangleleft a)
\triangleright b - (x \triangleleft b) \triangleright a$;\\
\item[(4)] $\{x,\, y \} \triangleleft a = - \{x,\, y \triangleleft
a\} - \{x \triangleleft a, \, y \} - x \triangleleft (y
\triangleright a) - y \triangleleft (x \triangleright a)$.
\end{enumerate}

Following \cite[Theorem 4.1]{majid} we introduce the following
concept:

\begin{definition}\delabel{mpmL}
Let $A = (A, [- , -])$ and $V = (V, \, \{-, -\})$ be two JJ
algebras. Then $(A, V, \triangleleft, \, \triangleright)$ is
called a \emph{matched pair} of JJ algebras if $(V, \,
\triangleleft)$ is a right JJ $A$-module, $(A, \,
\triangleright)$ is a left JJ $V$-module and the following
compatibilities hold for all $a$, $b \in A$, $x$, $y \in V$:
\begin{enumerate}
\item[(MP1)] $x \triangleright [a, \, b] = - [x \triangleright a,
\, b] - [a, \, x \triangleright b] - (x \triangleleft a)
\triangleright b - (x \triangleleft b) \triangleright a$;\\
\item[(MP2)] $\{x,\, y \} \triangleleft a = - \{x,\, y
\triangleleft a\} - \{x \triangleleft a, \, y \} - x \triangleleft
(y \triangleright a) - y \triangleleft (x \triangleright a)$.
\end{enumerate}
\end{definition}

If $(A, V, \triangleleft, \, \triangleright)$ is a matched pair of
JJ algebras then the associated unified product $A
\,\natural \,_{\Omega(A, V)} V$ will be denoted by $A \bowtie V $
and will be called the \emph{bicrossed product} of the matched
pair $(A, V, \triangleleft, \, \triangleright)$. Thus, $A \bowtie
V = A \times V$ as a vector space with multiplication given
by:
\begin{equation}\eqlabel{bracbicross}
[(a, x), \, (b, y)] := \bigl( [a, \, b] + x \triangleright b +
y\triangleright a , \,\, \{x, \, y \} + x\triangleleft b +
y\triangleleft a \bigl)
\end{equation}
for all $a$, $b \in A$ and $x$, $y \in V$.

\begin{example} \exlabel{semidirect}
Let $(A, V, \triangleleft, \, \triangleright)$ be a matched pair
of JJ algebras such that $\triangleleft$ is the trivial map.
Then the associated bicrossed product $A \bowtie V$ will be denoted by
$A \ltimes V$ and was first introduced in \cite{am-2015} under the name of
\emph{semidirect product}. Explicitly, the semidirect
product $A \ltimes V$ is associated to a left JJ $V$-module
structure $(A, \, \triangleright)$ such that for any $a$, $b\in A$
and $x\in V$:
$$
x \triangleright [a, \, b] = - [x \triangleright a, \, b] - [a, \,
x \triangleright b]
$$
or equivalently the map $ x \triangleright - : A \to A$ is an
antiderivation of $A$, for all $x\in V$.
\end{example}

The bicrossed product of two JJ algebras is the construction
responsible for the so-called \emph{factorization problem}:

\emph{Let $A$ and $V$ be two given JJ algebras. Describe and
classify all JJ algebras $E$ that factorize through $A$ and
$V$, i.e. $E$ contains $A$ and $V$ as JJ subalgebras such
that $E = A + V$ and $A \cap V = \{0\}$.}

Indeed, using \thref{classif} we can prove the JJ algebra version of
\cite[Theorem 3.9]{LW}:

\begin{corollary}\colabel{bicromlver}
A JJ algebra $E$ factorizes through two given JJ
algebras $A$ and $V$ if and only if there exists a matched pair of
JJ algebras $(A, V, \triangleleft, \, \triangleright)$ such
that $E \cong A \bowtie V $.
\end{corollary}

\begin{proof}
First observe that $A \cong A\times \{0\}$ and $V \cong
\{0\}\times V$ are JJ subalgebras of $A \bowtie V $ and of course $A
\bowtie V $ factorizes through $A\times \{0\}$ and $\{0\}\times V$.
Conversely, assume that a JJ algebra $E$ factorizes through
two JJ subalgebras $A$ and $V$. Since $V$ is a subalgebra of
$E$, the cocycle $f = f_p : V\times V \to A$ constructed in the
proof of \thref{classif} is just the trivial map $f_p (x, y) = 0$,
for all $x$, $y \in V$. Thus, the unified product $A \,\natural
\,_{\Omega(A, V)} V = A \bowtie  V $ coincides with the bicrossed product of
the JJ algebras $A$ and $V:= {\rm Ker}(p)$.
\end{proof}

Based on \coref{bicromlver} we can restate the factorization problem as follows: \emph{Let $A$ and $V$
be two given JJ algebras. Describe the set of all matched
pairs $(A, V, \, \triangleleft, \, \triangleright)$ and classify
up to an isomorphism all bicrossed products $A \bowtie V$}.

Due to its important applications to the theory of JJ algebras we will consider this problem separately in greater detail in a forthcoming paper.

In what follows we compute the Galois group of the JJ algebra
extension $A \subseteq A \bowtie V$. Given a matched pair
of JJ algebras $(A, V, \triangleleft, \, \triangleright)$ we define the \emph{Galois group} ${\rm
Gal} \, ( A \bowtie V / A)$ of the extension $A \subseteq A
\bowtie V$, as the subgroup of ${\rm Aut}_{{\rm JJ}} (A \bowtie V)$ of all
JJ algebra automorphisms of $A \bowtie V$ that stabilize
$A$:
$$
{\rm Gal} \, (A \bowtie V / A ) := \{ \sigma \in {\rm Aut}_{{\rm JJ}}(A
\bowtie V) \, | \, \sigma (a) = a, \, \forall \, a\in A \}
$$
As a straightforward consequence of \leref{morfismuni} we obtain a bijection
between the set of all elements $\psi \in {\rm Gal} \, ( A \bowtie
V / A )$ and the set of all pairs $(\sigma, r) \in {\rm GL}_k (V)
\times {\rm Hom}_k (V, \, A)$, satisfying the following compatibility conditions for any $a \in A$, $x$, $y \in V$:
\begin{enumerate}
\item[(G1)] $v(x) \triangleleft  a = v(x \triangleleft a)$;
\item[(G2)] $ v(x) \triangleright  a = r(x \triangleleft a) + x
\triangleright a - [a, \, r(x)]$; \item[(G3)] $v(\{x, y\}) =
\{v(x), v(y)\} + v(x) \triangleleft  r(y) + v(y) \triangleleft
r(x)$; \item[(G4)] $r(\{x, y\}) = [r(x), \, r(y)] + v(x)
\triangleright  r(y) + v(y) \triangleright  r(x) $
\end{enumerate}
The bijection is such that $\psi = \psi_{(\sigma, r)} \in {\rm
Gal} \, (A \bowtie V / A) $ corresponding to $(\sigma, r) \in {\rm
GL}_k (V) \times {\rm Hom}_k (V, \, A)$ is given by $\psi(a, \, x)
:= (a + r(x), \, \sigma(x))$, for all $a\in A$ and $x \in V$. We
point out that $\psi_{(\sigma, r)}$ is indeed an element of ${\rm
Gal} \, (A \bowtie V / A )$ with the inverse given by $
\psi_{(\sigma, r)}^{-1}(a, \, x) = \bigl(a - r(\sigma^{-1}(x)), \,
\sigma^{-1}(x)\bigl)$, for all $a\in A$ and $x \in V$.

We denote by ${\mathbb G}_{A}^V \, \bigl( \triangleleft, \,
\triangleright \bigl)$ the set of all pairs $(\sigma, \, r) \in
{\rm GL}_k (V) \times {\rm Hom}_k (V, \, A)$ satisfying the
compatibility conditions (G1)-(G4). It can be easily seen
that ${\mathbb G}_{A}^V \, \bigl( \triangleleft, \, \triangleright
\bigl)$ is a subgroup of the semidirect product of groups ${\rm
GL}_k (V) \rtimes {\rm Hom}_k (V, \, A)$ with the group structure
defined as follows:
\begin{equation} \eqlabel{grupstrc}
(\sigma, \, r ) \cdot (\sigma', \, r') := (\sigma \circ \sigma',
\, r \circ \sigma' + r')
\end{equation}
for all $\sigma$, $\sigma' \in {\rm GL}_k (V)$ and $r$,
$r'\in {\rm Hom}_k (V, \, A)$.
Now, for any $(\sigma, \, r)$ and $(\sigma', \, r') \in {\mathbb
G}_{A}^V \, \bigl( \triangleleft, \, \triangleright \bigl)$, $a
\in A$ and $x\in V$ we have:
$$
\psi_{(\sigma, \, r)} \circ \psi_{(\sigma', \, r')} (a, \, x) =
\bigl(a + r'(x) + r(\sigma' (x) ), \, \sigma (\sigma'(x) \bigl) =
\psi_{(\sigma\circ \sigma', \, r\circ \sigma' + r')} (a, \, x)
$$
i.e. $\psi_{(\sigma, \, r)} \circ \psi_{(\sigma', \, r')} =
\psi_{(\sigma\circ \sigma', \, r\circ \sigma' + r')}$.  To summarize, we have proved the following:

\begin{corollary}\colabel{grupulgal}
Let $(A, V, \triangleleft, \, \triangleright)$ be a matched pair
of JJ algebras. Then there exists an isomorphism of groups
defined as follows: \begin{equation} \eqlabel{izogal}
\Omega: {\mathbb G}_{A}^V \, \bigl( \triangleleft, \,
\triangleright \bigl) \, \to {\rm Gal} \, (A \bowtie V / A), \quad
\Omega (\sigma, r) \bigl((a, \, x)\bigl) := \bigl(a + r(x), \,
\sigma (x) \bigl)
\end{equation}
for all $(\sigma, \, r) \in {\mathbb G}_{A}^V \, \bigl(
\triangleleft, \, \triangleright \bigl)$, $a\in A$ and $x\in V$. In particular, there exists an embedding ${\rm Gal} \, (A \bowtie
V / A) \hookrightarrow {\rm GL}_k (V) \rtimes {\rm Hom}_k (V, \,
A)$, where the right hand side is the semidirect product of groups
defined by \equref{grupstrc}.
\end{corollary}

\subsection*{Crossed products and supersolvable algebras}
Let $\Omega(A, \, V) = \bigl(\triangleleft, \, \triangleright, \,
f, \{-, \, -\} \bigl)$ be an extending datum of the JJ algebra
$A$ through a vector space $V$ such that $\triangleleft$ is
trivial, i.e. $x \triangleleft a = 0$, for all $x\in V$ and $a
\in A$. Then, $\Omega(A, \, V) = \bigl(\triangleright, \, f, \{-,
\, -\} \bigl) $ is a JJ extending structure of $A$ through
$V$ if and only if $(V, \{-, -\})$ is a JJ algebra and the
following compatibilities hold for all $a$, $b\in A$ and $x$, $y$,
$z\in V$:
\begin{enumerate}
\item[(CP1)] $f\colon V\times V \to A$ is a symmetric map;\\
\item[(CP2)] $x \triangleright [a, \, b] = - [x \triangleright a,
\, b] - [a, \, x \triangleright b]$, i.e.
$x \triangleright - : A \to A$ is an antiderivation of $A$; \\
\item[(CP3)] $\{x,\, y \} \triangleright a =  - x \triangleright
(y \triangleright a) - y \triangleright (x \triangleright a) - [a,
\,
f(x, \, y)]$; \\
\item[(CP4)] $\sum_{(c)} \, f(x, \{y, z \}) + \sum_{(c)} \,
x\triangleright f(y, z) = 0$
\end{enumerate}
A system $(A, V, \triangleright, f)$ consisting of two JJ
algebras $A$, $V$ and two bilinear maps $\triangleright : V \times
A\to A$, $f: V\times V \to A$ satisfying the above four
compatibility conditions was called a \emph{crossed system} of $A$ and $V$ in \cite[Proposition
2.2]{am-2015}. In this
case, the associated unified product $A \,\natural \,_{\Omega(A,
V)} V = A \#_{\triangleright}^f \, V $ is the \emph{crossed
product} of the JJ algebras $A$ and $V$ and is defined as
follow: $A \#_{\triangleright}^f \, V = A \times \, V $ with the
multiplication given for any $a$, $b \in A$ and $x$, $y \in V$ by:
\begin{equation}\eqlabel{brackcrosspr}
[(a, x), \, (b, y)] := \bigl( [a, \, b] + x \triangleright b +
y\triangleright a + f(x, y), \, \{x, \, y \} \bigl)
\end{equation}

If $(A, V, \triangleright, f)$ is a crossed system of two JJ
algebras then, $A \cong A \times \{0\}$ is an ideal in $A
\#_{\triangleright}^f \, V$ since $[(a, 0), \, (b, y)] := \bigl(
[a, \, b] + y\triangleright a , \, 0 \bigl)$. Conversely, crossed
products describe all JJ algebra structures on a vector
space $E$ such that a given JJ algebra $A$ becomes an ideal of
$E$.

\begin{corollary}\colabel{croslieide}
Let $A$ be a JJ algebra, $E$ a vector space containing $A$
as a subspace. Then any JJ algebra structure on $E$ that
contains $A$ as an ideal is isomorphic to a crossed product of
JJ algebras $A \#_{\triangleright}^f \, V$.
\end{corollary}

\begin{proof}
Let $[-,\, -]$ be a JJ algebra structure on $E$ such that
$A$ is an ideal in $E$. In particular, $A$ is a subalgebra of $E$
and hence we can apply \thref{classif}. In this case the action
$\triangleleft = \triangleleft_p$ of the JJ extending
structure $\Omega(A, \, V) = \bigl(\triangleleft_p, \,
\triangleright_p, \, f_p, \{-, \, -\}_p \bigl)$ constructed in the
proof of \thref{classif} is trivial since for any $x \in
V$ and $a \in A$ we have $[x, a] \in A$ and hence $p ([x, a])
= [x, a]$. Thus, $x \triangleleft_p a = 0$, i.e. the unified
product $A \,\natural \,_{\Omega(A, V)} V = A
\#_{\triangleright}^f \, V $ is the crossed product of the
JJ algebras $A$ and $V:= {\rm Ker}(p)$.
\end{proof}

The crossed product of JJ algebras was studied in
detail in \cite{am-2015} related to Hilbert's extension
problem. We consider here a new application: using
\coref{croslieide} we show that crossed products play a key role in the classification of finite dimensional supersolvable JJ
algebras.

\begin{definition} \delabel{supersov}
Let $n$ be a positive integer. An $n$-dimensional JJ algebra $E$ is called
\emph{supersolvable} if there exists a finite chain of ideals of
$E$
\begin{equation} \eqlabel{lant2}
0 = I_0  \subset I_1 \subset \cdots \subset I_n = E
\end{equation}
such that $I_j$ has codimension $1$ in $I_{j+1}$, for all $j = 0,
\cdots, n-1$.
\end{definition}

All finite dimensional supersolvable JJ algebras can be
classified by a recursive method based on \coref{croslieide}. Indeed, the key step of the process consist in describing all crossed
products $A \, \#_{\triangleright}^f \, V$, for a $1$-dimensional
vector space $V$ and a given JJ algebra $A$. To this end, we introduce the following:

\begin{definition} \delabel{supersov}
Let $A$ be a JJ algebra. A \emph{supersolvable datum} of $A$
is a pair $(D, \, a_0) \in {\rm End}_k (A) \times A$ such that:
\begin{enumerate}
\item[(S1)] $D \colon A \to A$ is an antiderivation of $A$ and $ 3
D(a_0) =
0$; \\
\item[(S2)] $2 D^2 (a) = - [a, \, a_0]$, for all $a\in A$.
\end{enumerate}
We denote by ${\mathcal S} (A)$ the set of all supersovable data
of $A$.
\end{definition}

\begin{proposition}\prlabel{crossdim1}
Let $k$ be a field of characteristic $\neq 3$, $A$ a JJ
algebra and $V$ a vector space of dimension $1$ with basis
$\{x\}$. Then there exists a bijection between the set of all
crossed systems of $A$ and $V$ and the set ${\mathcal S} (A)$ of
all supersolvable data of $A$. Through the above bijection, the
crossed system $ \bigl( \triangleright, f, \{-, -\} \bigl)$
corresponding to $(D, \, a_0) \in {\mathcal S} (A)$ is defined as follows:
\begin{equation}\eqlabel{extenddim1a}
x \triangleright a = D(a), \quad f (x, x) = a_0, \quad \{x, \, x
\} = 0
\end{equation}
for all $a \in A$.
\end{proposition}

\begin{proof} Since $k$ is a field of characteristic $\neq 3$, the only JJ algebra structure on $V := kx$ is the abelian one,
i.e. $\{x, \, x \} = 0$. Moreover, as $V$ has dimension $1$ the
set of all bilinear maps $\triangleright : V \times A \to A$, $f:
V\times V \to A$ is in bijection with the set of all pairs $(D, \,
a_0) \in {\rm End}_k (A) \times A$ and the bijection is given such
that \equref{extenddim1a} holds. We are left to prove that the compatibilities (CP2)-(CP4) are equivalent to $(D, \, a_0,) \in {\mathcal S} (A)$. Indeed, (CP2) is
equivalent to the fact that $D$ is an antiderivation of $A$, (CP3)
is equivalent to the fact that (S2) holds and (CP4) is equivalent
to $ 3 D(a_0) = 0$.
\end{proof}

Now let $(D, \, a_0) \in {\mathcal S} (A) $ be a supersolvable
datum of $A$. The crossed product $A \,\#_{\triangleright}^f \,
kx$ associated to the crossed system \equref{extenddim1a} will be
denoted by $A_{(D, \, a_0)} := A \times kx $ and has the
multiplication given as follows:
\begin{equation}\eqlabel{extenddim2022a}
[(a, x), \, (b, x)] = ([a, b] + D(a) + D(b) + a_0, \, 0]
\end{equation}
for all $a$, $b\in A$. Using \coref{croslieide} and \prref{crossdim1} we obtain:

\begin{corollary}\colabel{balama}
Let $k$ be a field of characteristic $\neq 3$ and $A$ a JJ
algebra. Then a JJ algebra $E$ contains $A$ as an ideal of
codimension $1$ if and only if there exists a pair $(D, \, a_0)
\in {\mathcal S} (A)$ such that $E \cong A_{(D, \, a_0)}$.
\end{corollary}

\begin{remark} \relabel{ivan}
The concept of supersolvable JJ algebras was introduced by analogy with the Lie algebra theory. In that context, the classical Lie theorem proves that 
over an algebraically closed field of characteristic zero, any finite dimensional solvable Lie algebra is supersolvable. Consequently, we ask the following question:  

\emph{Let $k$ be an algebraically closed field of characteristic zero and $A$ a finite dimensional solvable JJ algebra. Is $A$ supersolvable?}

One of the reasons for asking this question is the following: Zelmanov and Skosyrskii \cite[Theorem 1 and Corollary 1]{zel} proved that \emph{any JJ algebra $A$ without elements of order $\leq 5$ in its additive group $(A, +)$ is solvable. In particular, if $k$ is a field of characteristic $\neq 2$, $3$ or $5$, then any JJ algebra is solvable.}

Now, a positive answer to the above question collaborated with the aforementioned result will show that,  
over an algebraically closed field of characteristic zero, any finite dimensional JJ algebra is supersolvable. Therefore, the classification of all finite dimensional JJ algebras could be reduced to the recursive method described above.  
\end{remark}

\subsection*{Skew crossed products and an Artin type theorem}
Let $A$ be a JJ algebra, $V$ a vector space and $\Omega(A, \, V) = \bigl(\triangleleft, \, \triangleright, \,
f, \{-, \, -\} \bigl)$ an extending datum of $A$ through $V$ such that $\triangleright$ is
trivial, i.e. $x \triangleright a = 0$, for all $x\in V$ and
$a \in A$. Then, using \thref{1} we obtain that $\Omega(A, \, V) =
\bigl(\triangleleft, \, f, \{-, \, -\} \bigl) $ is a JJ
extending structure of $A$ through $V$ if and only if the
following compatibilities hold for all $a$, $b\in A$ and $x$, $y$,
$z\in V$:
\begin{enumerate}
\item[(SC1)] $f: V\times V \to A$ and $\{-, \, -\} : V\times V
\to V$ are symmetric maps; \\
\item[(SC2)] $(V, \, \triangleleft)$ is a right JJ $A$-module;\\
\item[(SC3)] $\{x,\, y \} \triangleleft a = - \{x,\, y
\triangleleft a\} - \{x \triangleleft a, \, y \}$;\\
\item[(SC4)] $ [a, \, f(x, \, y)] + f(x, \, y \triangleleft a) + f(x \triangleleft a, \, y) = 0 $;\\
\item[(SC5)] $ \sum_{(c)} \, f\bigl(x, \{y,\, z \}\bigl) = 0$;\\
\item[(SC6)] $ \sum_{(c)} \, \{x, \, \{y, \,z\}\} + \sum_{(c)} \,
x \triangleleft f(y, z) = 0$.
\end{enumerate}
A system $(A, V, \triangleleft, f)$ satisfying the above compatibility conditions will be called a \emph{skew crossed system}
of $A$ through $V$ while the associated unified product $A
\,\natural \,_{\Omega(A, V)} V$ will be denoted by $A \,
\#^{\bullet} \, V $ and called the \emph{skew crossed product} of
$A$ and $V$. Thus $A \, \#^{\bullet} \, V  = A \times \, V $ with
the multiplication given as follows:
\begin{equation}\eqlabel{skewpr}
[(a, x), \, (b, y)] := \bigl( [a, \, b] + f(x, y), \, \{x, \, y \}
+ x \triangleleft b + y\triangleleft a \bigl)
\end{equation}
for all $a$, $b \in A$ and $x$, $y \in V$.

Skew crossed products will be used to prove an Artin type theorem for
JJ algebras, which provides a way of reconstructing a JJ
algebra from its subalgebra of invariants:

\begin{theorem} \thlabel{recsiGal}
Let $G$ be a finite group of invertible order in $k$ acting on a
JJ algebra $A$ by means of a group morphism $\varphi: G \to {\rm
Aut}_{\rm JJ} (A)$, $\varphi (g) (a) = g\cdot a$, for all $g\in G$
and $a\in A$. Let $A^G := \{a \in A \, | \, g\cdot a = a, \forall
g\in G \} \subseteq A$ be the subalgebra of invariants and $V$ a
complement of $A^G$ in $A$.

Then there exists a skew crossed system $\Omega(A^G, \, V) =
\bigl(\triangleleft, \, f, \{-, \, -\} \bigl)$ of $A^G$ through $V$
and an isomorphism of JJ algebras $A \cong A^G \, \#^{\bullet} \, V$.
\end{theorem}

\begin{proof}
First we note that $A^G$ is indeed a subalgebra of $A$ and $g
\cdot [a, \, b] = [g\cdot a, \, g\cdot b]$, for all $g\in G$, $a$,
$b\in A$. We define the trace map as follows for all $x\in A$:
\begin{equation}\eqlabel{tracedef}
t : A \to A^G, \quad t (x) :=  |G|^{-1} \, \sum_{g\in G} \, g
\cdot x
\end{equation}
We observe that $t(x) \in A^G$ and for all $a \in A^G$ we obtain:
$$
t ( [x, \, a ] ) = |G|^{-1} \, \sum_{g\in G} \, g \cdot [x, \, a]
= |G|^{-1} \, \sum_{g\in G} \, [g\cdot x, \, g\cdot a] = [t(x), \,
a]
$$
Secondly, we note that the trace map $t : A \to
A^G$ is a linear retraction of the canonical inclusion $A^G
\hookrightarrow A$, i.e. $t (a) = a$, for all $a\in A^G$. Now, if
we compute the canonical extending structure of $A^G$ through $V :=
{\rm Ker} (t)$ associated to the trace map $t$, using the formulas
from the proof of \thref{classif}, we obtain that for all $x \in
V$ and $a \in A^G$ we have:
$$
x\triangleright_t a = t([x, \, a]) = [t(x), \, a] = 0,
$$
i.e. the action $\triangleright_t$ is the trivial one. Thus, the
extending structure of $A^G$ through $V$ associated to the trace
map $t$ reduces to a skew crossed system and applying once again
\thref{classif} we obtain that the map defined for all $a\in A^G$
and $x\in V$ by:
\begin{equation} \eqlabel{recizoaa}
\vartheta: A^G \, \#^{\bullet} \, V \to A, \qquad \vartheta(a, x)
:= a + x
\end{equation}
is an isomorphism of JJ algebras. This finishes the proof.
\end{proof}

\begin{remark} \relabel{hilbert}
In the context of \thref{recsiGal}, if $G = <g> $ is a finite
cyclic group generated by $g$, then we can easily prove that the
kernel of the trace map defined by \equref{tracedef} has a nice
description, namely: ${\rm Ker} (t) = \{ a - g\cdot a \, | \, a
\in A\}$.
\end{remark}

\section{Flag extending structures}\selabel{exemple}

\thref{main1} offers the theoretical answer to the extending
structures problem. However, computing the classifying object
${\mathcal H}^{2}_{A} \, (V, \, A)$, for a given JJ algebra
$A$ and a vector space $V$ is a highly non-trivial task. In
this section we shall explicitly compute ${\mathcal H}^{2}_{A} \,
(k, \, A)$.

\begin{definition} \delabel{flagex}
Let $A$ be a JJ algebra and $E$ a vector space containing
$A$ as a subspace. A JJ algebra structure on $E$ such that
$A$ is a subalgebra is called a \emph{flag extending structure} of
$A$ to $E$ if there exists a finite chain of subalgebras of $E$
\begin{equation} \eqlabel{lant}
E_0 := A  \subset E_1 \subset \cdots \subset E_m = E
\end{equation}
such that $E_i$ has codimension $1$ in $E_{i+1}$, for all $i = 0,
\cdots, m-1$. 
\end{definition}

All flag extending structures of $A$ to $E$ can be completely described and classified
by a recursive method. The key step of this process is the case $m = 1$ which
describes and classifies all unified products $A \,\natural \,
V_1$, for a $1$-dimensional vector space $V_1$. We can now continue the process by replacing
the initial JJ algebra $A$ with such a unified product $A
\,\natural \, V_1$. The latter product can be described in terms of $A$ only and iterating the process we obtain the description of all flag extending structures of $A$ to
$E$ after $m = {\rm dim}_k (V)$ steps. We start by introducing the following concept:

\begin{definition}\delabel{lambdaderivariii}
A \emph{flag datum} of a JJ algebra $A$ is a system $(D, \,
\lambda, \,  a_0, \, \alpha_0) \in {\rm End}_k (A) \times A^*
\times A \times k$ such that:
\begin{enumerate}
\item[(F1)] $\lambda ([a, \, b]) + 2 \lambda(a)\lambda(b) = 0 $; \\
\item[(F2)] $D([a, \, b]) = -[D(a),\, b] - [a, \, D(b)] - \lambda(a) D(b) - \lambda(b) D(a)$ ;\\
\item[(F3)] $[a, \, a_0] + \alpha_0 D(a) + 2 D^2 (a) + 2 \lambda(a) a_0 = 0$; \\
\item[(F4)] $3 \lambda(a) \alpha_0 + 2 \lambda(D(a)) = 0 $;\\
\item[(F5)] $3 D(a_0) + 3 \alpha_0 a_0  = 0$;\\
\item[(F6)] $3 \lambda(a_0) + 3 \alpha_0 ^2 = 0$.
\end{enumerate}
for all $a$, $b\in A$. The set of all flag data of $A$ will be
denoted by ${\mathcal F} (A)$.
\end{definition}

\begin{examples} \exlabel{abelflg}
1. If $D$ is an antiderivation of $A$ with $D^2 = 0$, then $(D, \,
\lambda: = 0, \,  a_0:= 0, \, \alpha_0 := 0)$ is a flag datum of
$A$.

2. Assume that $k$ is a field of characteristic $ \neq 2, 3$ and
let $A$ be the abelian JJ algebra, i.e. $[a, \, b] = 0$, for
all $a$, $b\in A$. Then the set ${\mathcal F} (A)$ of all flag
data of $A$ is in bijection with the set of all pairs $(D, \, a_0)
\in {\rm End}_k (A) \times A$, such that $D^2 = 0$ and $D(a_0) =
0$.
\end{examples}

We shall prove now that the set of all JJ extending
structures ${\mathcal J} {\mathcal J} \, (A, \, V)$ of a JJ
algebra $A$ through a $1$-dimensional vector space $V$ is
parameterized by ${\mathcal F} (A)$.

\begin{proposition}\prlabel{unifdim1}
Let $A$ be a JJ algebra and $V$ a vector space of dimension
$1$ with basis $\{x\}$. Then there exists a bijection between
the set ${\mathcal J} {\mathcal J} \, (A, \, V)$ of all JJ
extending structures of $A$ through $V$ and the set ${\mathcal F}
(A)$ of all flag data of $A$. Through the above bijection, the
JJ extending structure $\Omega(A, \, V)  =
\bigl(\triangleleft, \, \triangleright, f, \{-, -\} \bigl)$
corresponding to $(D, \, \lambda, \,  a_0, \, \alpha_0) \in
{\mathcal F} (A)$ is given for all $a \in A$ by:
\begin{equation}\eqlabel{extenddim1}
x \triangleleft a = \lambda (a) x, \quad x \triangleright a =
D(a), \quad f (x, x) = a_0, \quad \{x, \, x \} = \alpha_0 \, x
\end{equation}
\end{proposition}

\begin{proof}
Since $V := kx$ has dimension $1$, the set of all bilinear maps
$\triangleleft : V \times A \to V$, $\triangleright : V \times A
\to A$, $f: V\times V \to A$ and $\{-, \, -\} : V\times V \to V$
is in bijection with the set of all systems $(D, \, \lambda, \,
a_0, \, \alpha_0) \in {\rm End}_k (A) \times A^* \times A \times
k$ and the bijection is given such that \equref{extenddim1} hold.
The only thing left to prove is that the compatibility conditions
(E1)-(E7) from \thref{1} are equivalent to (F1)-(F6) from
\deref{lambdaderivariii}. This follows by a straightforward
computation.
\end{proof}

\begin{remark}\relabel{unifbicross}
Let $(D, \, \lambda, \,  a_0, \, \alpha_0) \in {\mathcal F} (A)$
be a flag datum of $A$. The unified product $A \,\natural \, kx$
associated to the JJ extending structure \equref{extenddim1}
will be denoted by $A_{(D, \, \lambda, \,  a_0, \, \alpha_0)}$ and
has the multiplication given for all $a$, $b\in A$ by:
\begin{equation}\eqlabel{extenddim2022}
[(a, x), \, (b, x)] = ([a, b] + D(a) + D(b) + a_0, \, (\lambda(a)
+ \lambda(b) + \alpha_0) x]
\end{equation}

Explicitly, if $\{e_i \, | \, i\in I\}$ is a $k$-basis of $A$ then
$A_{(D, \, \lambda, \, a_0, \, \alpha_0)}$ is the JJ algebra
having $\{x, \, e_i \, | \, i\in I\}$ as a $k$-basis and the
multiplication given by:
\begin{equation}\eqlabel{extenddim20}
[e_i, \, e_j] := [e_i, \, e_j]_A, \quad [e_i, \, x] := D(e_i) +
a_0 + \lambda(e_i) \, x, \quad [x, \, x] := a_0 + \alpha_0 \, x
\end{equation}
\end{remark}

Using \thref{classif} and \prref{unifdim1} we obtain:

\begin{corollary}\colabel{balam}
Let $A$ be a JJ algebra. Then a JJ algebra $E$
contains $A$ as a subalgebra of codimension $1$ if and only if
there exists $(D, \, \lambda, \,  a_0, \, \alpha_0) \in {\mathcal
F} (A)$ a flag datum of $A$ such that $E \cong A_{(D, \, \lambda,
\,  a_0, \, \alpha_0)}$.
\end{corollary}

Now, an easy computation shows that the equivalence
relation defined in \deref{echiaa} applied for the set ${\mathcal F} (A)$ via formulas \equref{extenddim1} of
\prref{unifdim1} takes the following form:

\begin{definition} \delabel{echivtwderivari}
Two flag data $(D, \, \lambda, \,  a_0, \, \alpha_0)$ and $(D', \,
\lambda', \,  a_0', \, \alpha_0') \in {\mathcal F} (A)$ of a
JJ algebra $A$ are called \emph{equivalent} and we denote
this by $(D, \, \lambda, \,  a_0, \, \alpha_0) \equiv (D', \,
\lambda', \,  a_0', \, \alpha_0')$ if $\lambda = \lambda'$ and
there exists a pair $(r, u) \in A \times k^*$ such that for all $a
\in A$ we have:
\begin{eqnarray}
D(a) &=& u \, D'(a) + [a, r] - \lambda'(a) \, r \eqlabel{lzeci3a} \\
\alpha_0 &=& u \, \alpha_0' + 2 \lambda'(r) \eqlabel{lzeci2b} \\
a_0 &=& u^2 \, a'_0 + [r, \, r] + 2 u \, D' (r) - u \, \alpha_0'
\, r - 2 \lambda'(r) \, r  \eqlabel{lzeci1c}
\end{eqnarray}
\end{definition}

Next, we classify all JJ algebras $A_{(D, \, \lambda, \,
a_0, \, \alpha_0)}$ by computing the cohomological type object
${\mathcal H}^{2}_{A} (V, \, A)$, where $V$ is a $1$-dimensional
vector space. This is the first explicit classification result of
the extending structures problem for JJ algebras and the key
step in the classification of all flag extending structures.

\begin{theorem}\thlabel{clasdim1}
Let $A$ be a JJ algebra of codimension $1$ in the vector
space $E$ and $V$ a complement of $A$ in $E$. Then, $\equiv$ is an
equivalence relation on the set ${\mathcal F} (A)$ of all flag
data of $A$ and
$$
{\rm Extd} \, (E, \, A) \cong {\mathcal H}^{2}_{A} (V, \, A)
\cong {\mathcal F} (A)/\equiv
$$
The bijection between ${\mathcal F} (A)/\equiv$ and ${\rm Extd} \,
(E, \, A)$ is given by:
$$
\overline{ (D, \, \lambda, \,  a_0, \, \alpha_0) } \mapsto A_{(D,
\, \lambda, \, a_0, \, \alpha_0)}
$$
where $\overline{ (D, \, \lambda, \,  a_0, \, \alpha_0) }$ is the
equivalence class of $(D, \, \lambda, \,  a_0, \, \alpha_0)$ via
the relation $\equiv$ from \deref{echivtwderivari} and $A_{(D, \,
\lambda, \, a_0, \, \alpha_0)}$ is the JJ algebra
constructed in \equref{extenddim20}.
\end{theorem}

\begin{proof} Follows from \thref{main1} and the results of this section.
\end{proof}

\begin{remark}
In fact, the method described above can be used for classifying all finite dimensional JJ algebras over a field of characteristic $\neq 2$, $3$ or $5$. Indeed, using \cite[Corollary 1]{zel} we obtain that, over such a field, any JJ algebra is solvable and in particular is a flag extending structure of $\{0\}$. 
Thus, by starting the recursive method described above with $A := 0$ yields the description of all JJ algebras of a given finite dimension.
\end{remark} 

Next we provide two explicit examples for the above results by
computing ${\mathcal H}^{2}_{A} (V, \, A)$ and then describing all
JJ algebra structures which extend the JJ algebra
structure from $A$ to a vector space of dimension $1 + {\rm dim}_k
(A)$. The detailed computations are rather long but
straightforward and will be omitted.

\begin{example}
Let $k$ be a field of characteristic $\neq 2, 3$ and $A := k^n$,
the abelian JJ algebra of dimension $n$, i.e. $[a, \, b] = 0$, for all $a$,
$b\in k^n$. Then
$$
{\mathcal H}^{2}_{k^n} (k, \, k^n) \cong \{ (D, \, a_0) \in {\rm
End}_k (k^n) \times k^n \, | \, D^2 = 0, \,\, D(a_0) = 0 \}/\equiv
$$
where $(D, \, a_0) \equiv (D', \, a_0')$ if and only if there
exists a pair $(r, \, u) \in k^n \times k^*$ such that
$$
D(a) = u \, D'(a), \qquad a_0 = u^2 \, a_0' + 2u \, D'(r)
$$
for all $a\in k^n$. Indeed, using \exref{abelflg} the set of all
flag data ${\mathcal F} (k^n)$ identifies with the set $\{ (D,
\, a_0) \in {\rm End}_k (k^n) \times k^n \, | \, D^2 = 0, \,\,
D(a_0) = 0 \}$. The conclusion follows from \thref{clasdim1}.

Let $\{e_i \, | \, i = 1, \cdots, n \}$ be the canonical basis of
$k^n$ and $(D, \, a_0) \in {\mathcal F} (k^n)$. Then $k^n_{(D, \,
a_0)}$ is the $(n+1)$-dimensional JJ algebra with
multiplication given for all $i$, $j = 1, \cdots, n$:
$$
[e_i, \, e_j] := 0, \quad [e_i, \, e_{n+1}] = [e_{n+1}, \, e_{i}]
:= D(e_i), \quad [ e_{n+1}, \,  e_{n+1}] := a_0
$$
Any $(n+1)$-dimensional JJ algebra containing $k^n$ as an
abelian subalgebra is isomorphic to $k^n_{(D, \, a_0)}$.
\end{example}

\begin{example}
Let $k$ be a field of characteristic $\neq 2, 3$ and $A :=
\mathfrak{h} (3, k)$ the Heisenberg JJ algebra \cite{BF}
having $\{e_1, e_2, e_3 \}$ as a $k$-basis and the multiplication
defined by $[e_1, \, e_2] = [e_2, \, e_1] = e_3$.

Then, a straightforward computation shows that
the set of all flag data ${\mathcal F} (\mathfrak{h} (3, k))$ is in bijection with the set of triples $(\alpha,\, \beta,\, \gamma) \in k^{3}$
which satisfy $\alpha\, \gamma= 0$.

The bijection is defined such that the flag datum $(D, \, \lambda, \,  a_0, \, \alpha_0) \in
{\mathcal F} (\mathfrak{h} (3, k))$ corresponding to $(\alpha,\, \beta,\, \gamma)$ is given by:
\begin{eqnarray*}
&& \lambda \equiv 0,\,\,\, \alpha_{0} = 0,\,\,\, a_{0} = \alpha \, e_{2}\\
&& D(e_{1}) = \beta \, e_{3},\,\,\, D(e_{2}) = \gamma\, e_{3},\,\,\, D(e_{3}) = 0
\end{eqnarray*}
The compatibility condition $\alpha\, \gamma= 0$ imposes a
discussion on whether $\alpha = 0$ or $\alpha \neq 0$. If $\alpha = 0$, the corresponding flag datum defines the family of JJ algebras denoted by $\mathfrak{h} (3, k)_{(\beta,\, \gamma)}$ with the following multiplication:
\begin{eqnarray*}
[e_1, \, e_2] = [e_2, \, e_1]  = e_3,\,\, [e_1, \, x] = [x,\, e_1] = \beta\,e_3,\,\, [e_2, \, x] = [x,\, e_2] = \gamma\,e_3
\end{eqnarray*}
It can be easily seen that two flag datums induced by $(\beta,\, \gamma)$ and respectively $(\beta',\, \gamma')$ are
equivalent in the sense of \deref{echivtwderivari} if and only if there exists $u \in k^{*}$ such that $\beta \, \gamma = \beta' \, \gamma' u^{2}$. We denote by $\approx_1$ the equivalence relation on $k \times k$ defined as follows:
$$
(\beta,\, \gamma) \approx_1 (\beta',\, \gamma') \,\, {\rm if\,\, and\,\, only\,\,if\,\, there\,\, exists\,\,}\,\, u \in k^{*}\,\, {\rm such\,\, that} \,\, \beta \, \gamma = \beta' \, \gamma' u^{2}
$$

If $\alpha \neq 0$, the corresponding flag datum defines the family of JJ algebras denoted by $\mathfrak{h} (3, k)_{\beta}^{\alpha}$ with the following multiplication:
\begin{eqnarray*}
&& [e_1, \, e_2] = [e_2, \, e_1]  = e_3,\,\, [e_1, \, x] = [x,\, e_1] = \beta\,e_3 + \alpha\, e_{2},\\
&& [e_2, \, x] = [x,\, e_2] = \alpha\,e_3,\,\, [x,\, x] = \alpha\, e_{2}
\end{eqnarray*}
Now two flag datums induced by $(\alpha,\, \beta)$ and respectively $(\alpha',\,\beta')$ are equivalent in the sense of \deref{echivtwderivari} if and only if there exists $u \in k^{*}$ such that $\alpha = \alpha' u^{2}$. In particular, this shows that given $\beta \in k$ the flag datum induced by $(\alpha,\, \beta)$ is equivalent to the flag datum induced by $(\alpha,\, 0)$. We denote by $\approx_2$ the equivalence relation on $k$ defined as follows:
$$
\alpha \approx_2 \alpha' \,\, {\rm if\,\, and\,\, only\,\,if\,\, there\,\, exists\,\,}\,\, u \in k^{*}\,\, {\rm such\,\, that}\,\, \alpha = \alpha' u^{2}
$$

Furthermore, a flag datum induced by a triple with $\alpha = 0$ is never equivalent to a flag datum induced by a triple with $\alpha \neq 0$. This leads to the
description of ${\mathcal H}^{2}_{\mathfrak{h} (3, k)} \bigl(k, \, \mathfrak{h} (3, k)
\bigl)$ as the following coproduct of sets:
$$
{\mathcal H}^{2}_{\mathfrak{h} (3, k)} \bigl(k, \, \mathfrak{h} (3, k)
\bigl) \cong \bigl((k \times k)/ \approx_1 \bigl) \sqcup \bigl(k^{*}/ \approx_2\bigl)
$$
\end{example}

\textbf{Acknowledgement:} The authors are grateful to Ivan Shestakov for its valuable comments which led to a substantial improvement of the paper.

\end{document}